\documentclass[12pt,centertags,oneside]{amsart} 
\usepackage{amsmath,amstext,amsthm,amscd,typearea} 
\usepackage[colorlinks,linkcolor=red, anchorcolor=black,citecolor=blue]{hyperref} 
\usepackage{amssymb} 
\usepackage{a4wide} 
\usepackage[mathscr]{eucal} 
\usepackage{mathrsfs} 
\usepackage{typearea} 
\usepackage{charter} 
\usepackage[normalem]{ulem} 
\usepackage{pdfsync} 
\usepackage{longtable} 
\usepackage{mathtools} 
\usepackage{bigdelim} 
\usepackage{color} 
\usepackage{enumerate} 
\usepackage{graphicx} 
\usepackage{mathrsfs} 
\usepackage{multirow} 
\usepackage[all]{xy} 
\usepackage{color} 
\usepackage{enumitem} 
\usepackage[a4paper,width=16.2cm,top=3cm,bottom=3cm]{geometry}  
 
\numberwithin{equation}{section}

\newtheorem{theorem}{Theorem}[section]
\newtheorem{proposition}[theorem]{Proposition}
\newtheorem{corollary}[theorem]{Corollary}
\newtheorem{lemma}[theorem]{Lemma}
\newtheorem{Proposition-Definition}[theorem]{Proposition-Definition}
\theoremstyle{definition}
\newtheorem{definition}[theorem]{Definition}

\newtheorem{problem}[theorem]{Problem}
\newtheorem{notation}[theorem]{Notation}
\newtheorem{remark}[theorem]{Remark}

\newcommand{\IP}{{\rm IP}}
\newcommand{\RP}{{\rm RP}}

\newcommand{\reg}{{\rm reg}}
\newcommand{\sing}{{\rm sing}}

\emergencystretch=\maxdimen
\title[Periodic points for self-maps of Fujiki varieties]{Periodic points for meromorphic self-maps of Fujiki varieties}

\author{Tien-Cuong Dinh}
\address{Department of Mathematics, National University 
of Singapore, 10 Lower Kent Ridge Road, Singapore 119076.}
\email{matdtc@nus.edu.sg}

\author{Guolei Zhong}
\address{Center for Complex Geometry,
	Institute for Basic Science (IBS),
	55 Expo-ro, Yuseong-gu, Daejeon, 34126, Republic of Korea}
\email{guolei@ibs.re.kr, zhongguolei@u.nus.edu}


\dedicatory{Dedicated to Professor De-Qi Zhang on the occasion of his 60th birthday}

\date{}
\begin{document}

\begin{abstract}
Let \(f\colon X\to X\) be a dominant meromorphic self-map of a compact complex variety \(X\) in the Fujiki class \(\mathcal{C}\).
If the topological degree of \(f\) is strictly larger than the other dynamical degrees of \(f\), we show that the number of isolated \(f\)-periodic points grows exponentially fast similarly to the topological degrees of the iterates of \(f\); in particular, we give a positive answer to a conjecture of Shou-Wu Zhang. 
In the general case, we show that the exponential growth of the number of isolated \(f\)-periodic points is at most the algebraic entropy of \(f\).
\end{abstract}

\maketitle


\section{Introduction}\label{s:intro}
As one of the classical problems in complex dynamics,
the equidistribution of orbits of points, analytic subsets, or periodic points is expected to impose strong ergodic properties on the dynamical system. 
Meanwhile, within a given dynamical system, the asymptotic upper bound of the number of isolated periodic points plays a significant role in their
equidistribution property. 
When the underlying space is a compact K\"ahler manifold, a lot of progress has been made in this direction over the past twenty years; 
we refer the reader to \cite{BD99,DNT-ind,DNT-BLMS,DNV-adv,DDG10,DS-survey,Duj06,Fav98, IU10,Kal00, KH07,Sai87,Xie15} and the references therein for the  related known results.

On the other hand, considering that an analytic periodic subset is not necessarily smooth, it becomes non-negligible to account for the presence of singularities when studying these sub-dynamical systems.
Motivated by a conjecture of Shou-Wu Zhang, in this paper, we study the growth of periodic points of a meromorphic map of a compact complex variety in the \textit{Fujiki class} \(\mathcal{C}\) (or \textit{Fujiki variety} for short).
A Fujiki variety is bimeromorphic to a compact K\"ahler manifold (cf.~Definition \ref{d:fujiki}).   
This is a continuation of the joint works of the first author \cite{DNT-ind} and \cite{DNT-BLMS}.

Let  \(f\colon X\to X\) be a meromorphic map of a compact complex variety of dimension \(k\) in the Fujiki class \(\mathcal{C}\).
We assume that \(f\) is dominant, i.e., the image of \(f\) contains a dense Zariski open subset of \(X\).
For a review of dynamical degrees, we refer to Definition \ref{d:dynamical-degree}. 
As the dynamical degrees for meromorphic maps on compact K\"ahler manifolds are bimeromorphic invariants (see \cite{DS-ens,Dinh-Sibony-Annals05}), given an integer \(0\leq p\leq k\), we can define the \(p\)-th dynamical degree \(d_p(f)\) of \(f\) by the \(p\)-th dynamical degree of the dominant meromorphic map on its smooth K\"ahler model induced by \(f\). 
We say that $f$ has a \textit{dominant topological degree} if its last dynamical degree \(d_t\coloneqq d_k(f)\) (also called the \textit{topological degree}) is strictly larger than the other dynamical degrees. 

Our first main theorem below provides an exact estimate of the asymptotic growth of periodic points if the meromorphic map has a dominant topological degree.

\begin{theorem} \label{t:main1}
Let $f\colon X\to X$ be a dominant meromorphic map of a compact complex variety in the Fujiki class \(\mathcal{C}\), and $Y$ a proper analytic subset of $X$ containing the indeterminacy set of $f$. 
Denote by $\IP_{n}^0(X\setminus Y,f)$ the set of all isolated \(f\)-periodic points of period $n$ whose orbits do not intersect $Y$. 
Assume that \(f\) has a dominant topological degree \(d_t\). 
Then we have the cardinality $\# \IP_n^0(X\setminus Y,f)=d_t^n(1+o(1))$ as $n$ tends to infinity.
\end{theorem}

\begin{remark}[{Multiplicities; cf.~Notation \ref{n:periodic}}]\label{r:mul}
When \(X\) is a compact K\"ahler manifold, Theorem \ref{t:main1} has been proved in \cite{DNT-ind}. 
In Theorem \ref{t:main1}, we refrain from counting multiplicities of the periodic points for the following reasons. 
First, the multiplicity of a singular point is not well-defined.  
Second, to get a compact K\"ahler model, apart from the resolution of singularities, one also needs to blow up certain smooth loci of \(X\); although the multiplicity of a periodic point 
within these smooth blown-up loci is well-defined, it is unclear whether their multiplicities could be controlled by their (positive-dimensional) inverse images on its  K\"ahler model \(\widetilde X\). 
However, if we count with multiplicities of the periodic points contained in the (open) smooth locus over which \(\widetilde X\to X\) is isomorphic, then Theorem \ref{t:main1} remains true; see Theorem \ref{t:equi-mero}.
Note that any modification of an analytic variety may change the number of periodic points. 
\end{remark}

To apply Theorem \ref{t:main1} for \(f\) holomorphic, 
we can take $Y$ to be empty, which enables us to give a positive answer to a conjecture of Shou-Wu Zhang \cite[Conjecture 1.2.2]{Zhang-distribution}. 

\begin{corollary}\label{c:swz}
Let \(f\colon X\to X\) be a polarized holomorphic self-map of a compact K\"ahler space, i.e., \(f^*[\omega]=q[\omega]\) for some K\"ahler class \([\omega]\) and some integer \(q>1\).
Then we have the cardinality \(\#\{x\in X~|~f^n(x)=x\}=q^{n\cdot\dim(X)}(1+o(1))\) as \(n\) tends to infinity. 
\end{corollary}
We refer to Definition \ref{d:cpt} and Remark \ref{r:2cpt}  for the review of compact K\"ahler spaces. 
Corollary \ref{c:swz} (cf.~Theorem \ref{t:equi-mero}) also implies that the periodic points of a polarized self-map of a compact K\"ahler space are countable and Zariski dense (cf.~\cite[Theorem 5.1]{Fak03} for the projective case). 

\begin{remark}[Holomorphic self-map with a dominant topological degree]\label{r:proj:dominant}
It's worth noting that when \(f\colon X\to X\) is a surjective holomorphic map of a projective variety \(X\), the condition of \(f\) having a dominant topological degree is equivalent to \(f\) being \textit{int-amplified}, i.e., \(f^*H-H\) is ample for some integral ample divisor \(H\) (see \cite[Lemma 3.6]{Men20} and \cite[Proposition 3.7]{MZ23}). 
We refer the reader to Proposition \ref{p:dominant-int} and Remark \ref{rmk:singular:dominant-int} for the case of compact K\"ahler spaces.
For more exploration on special holomorphic self-maps of compact K\"ahler spaces or projective varieties, we refer the reader to \cite{MZ18, Men20, MZ23, Zhong-Asian} and the references therein. 
\end{remark}

When the topological degree is not dominant, the problem of the growth of periodic points is still widely open.  
It has been proved in \cite{DNT-BLMS} that a dominant meromorphic self-map on a compact K\"ahler manifold is an \textit{Artin-Mazur} map (see~\cite{AM65}), i.e., the number of isolated periodic points of period \(n\) grows at most exponentially fast with \(n\); see Theorem \ref{t:per-kahler}. 
Our second main theorem extends \cite[Theorem 1.1]{DNT-BLMS} to the case of Fujiki varieties.
As indicated in Remark \ref{r:mul}, we do not count with multiplicities of the periodic points herein, either.

\begin{theorem}\label{t:main2}
Let \(f\colon X\to X\) be a dominant meromorphic map of a compact complex variety in the Fujiki class \(\mathcal{C}\) and \(I_f\) the indeterminacy set of \(f\). 
Denote by \(\IP_n^0(X\setminus I_f,f)\) the set of all isolated \(f\)-periodic points of period \(n\) whose orbits do not intersect \(I_f\).
Then we have
\[
\limsup_{n\to\infty}\frac{1}{n}\textup{log}\,\#\IP_n^0(X\setminus I_f,f)\leq h_a(f).
\]
Here, \(h_a(f)\) is the algebraic entropy of \(f\) (see Definition \ref{d:dynamical-degree}).  
In particular, the number of isolated \(f\)-periodic points of period \(n\) grows at most exponentially fast with respect to \(n\).
\end{theorem}
In comparison, we refer to Remark \ref{rem:other-alg-closed} for projective varieties in any algebraically closed field and Remark \ref{rem-fujiki-surfaces} for a precise estimate in the case of bimeromorphic maps of Fujiki surfaces.

\begin{remark}[Strategy towards the proofs]
Let us briefly explain our strategy to prove Theorems \ref{t:main1} and \ref{t:main2}. 
The approach to establishing the lower bound in Theorem \ref{t:main1} involves proving an equidistribution of repelling points outside a proper analytic subset \(Y\) (see Theorem \ref{t:equi-mero}). 
In this context, we follow a similar strategy as in \cite[Section 3]{DNT-ind} to construct an almost maximal number of inverse branches for a generic small ball with an additional property that they are disjoint with any given proper analytic subset \(Y\). 
In terms of the upper bound for both Theorems \ref{t:main1} and \ref{t:main2}, we first establish a comparison proposition that reveals the relation between the dynamical degrees of a dominant meromorphic map of a Fujiki variety and the dynamical degrees of its restriction to any proper analytic periodic subset (see Proposition \ref{p:small-dyn}). Then we control the growth of the intersection of periodic points with any proper analytic subset.

Since we are working with a singular Fujiki variety and the periodic points are not necessarily isolated, the classical tools like the Lefschetz fixed point formula could not be applied. 
Besides, our current approach to obtain the upper bound requires some new arguments together with results in \cite{DNT-ind} and \cite{DNT-BLMS}. 
The latter uses a very recent theory of density of positive closed currents developed by Sibony and the first author.  
\end{remark}

We note that the current techniques towards Theorem \ref{t:main1} in this paper do not allow us to answer the following natural question.
\begin{problem}
With the same notion and assumption as in Theorem \ref{t:main1}, is there a positive number \(\delta<d_t(f)\) such that 
\(\#\IP_n^0(X\setminus Y,f)=d_t^n+O(\delta^n)\) as \(n\) tends to infinity?
\end{problem}

\begin{remark}[Difficulties in the proofs]
The following are the two difficulties we met.
First, when we modify the given Fujiki variety to get a K\"ahler model, a priori, the number of isolated periodic points may be increased or decreased; in view of Remark \ref{r:mul}, even if \(f\) is a holomorphic self-map of a Fujiki manifold, we could not simply reduce the problem to its K\"ahler model. 
Second, a dominant meromorphic map of a Fujiki variety may no longer be dominant on subvarieties we consider (see Remark \ref{r:no product formula}).
\end{remark}

After preparing some preliminary results in Section \ref{s:pre}, we prove Theorem \ref{t:main2} in Section \ref{s:main2} and Theorem \ref{t:main1} in Section \ref{s:proofs of conj}. 
Let us end the section by summarizing the notations involving the periodic points.
\begin{notation}\label{n:periodic}
Let $f\colon X\to X$ be a dominant meromorphic map on a compact complex variety \(X\) of dimension \(k\), and  \(I_f\) the (first) indeterminacy set of \(f\). 
Denote by \(X_\reg\) (resp. \(X_\sing\)) the set of smooth points (resp. singular points) of \(X\).
\begin{enumerate}
\item An \textit{isolated \(f\)-periodic point} of period $n$ is a point $a\in X$ such that $(a,a)$ is an isolated intersection of the graph $\Gamma_n$ of $f^n$ (which is an irreducible analytic subset of dimension $k$ in $X\times X$) with the diagonal $\Delta$ of \(X\times X\). 
\item Let \(a\in X_\reg\) be an \(f\)-periodic smooth  point of period \(n\).  
Then its \textit{multiplicity} as an \(f\)-periodic point  is the multiplicity of the intersection of \(\Gamma_n\) and \(\Delta\) at \((a,a)\). 
\item Let \(a\in X_\reg\) be an \(f\)-periodic smooth point of period \(n\). 
We say that \(a\in X\) is a \textit{repelling point} of period \(n\) if \(a\not\in I_{f^n}\) holds, and all the eigenvalues of the differential of \(f^n\) at \(a\) have modulus greater than 1. 
Clearly, a repelling periodic point is of multiplicity 1. 
\end{enumerate} 
We use the following notation throughout this paper. 
We shall omit the symbol \(f\) if no confusion arises. 
Let \(A\) be a subset of \(X\). 
\begin{longtable}{p{1cm} p{1cm} p{13cm}}
$\IP_n(f)$ && the set of all the \(f\)-periodic points of period \(n\) without counting with multiplicities\\
$\IP_n(A,f)$ && the set of all the \(f\)-periodic points of period \(n\) in \(A\) without counting with multiplicities; these orbits are not necessarily contained in \(A\).\\
$\IP_n^0(A,f)$ && the set of all the \(f\)-periodic points of period \(n\) without counting with multiplicities whose orbits are well-defined and contained in \(A\)\\
$\widetilde \IP_n(f)$ &&the set of all the \(f\)-periodic smooth points of period \(n\) counted with multiplicities \\
$\RP_n(f)$ && the set of all the repelling points of period \(n\) \\
$\RP_n(A,f)$ &&the set of all the repelling points of period \(n\) in $A$\\
$\RP^0_n(A,f)$ && the set of all the repelling points of period \(n\) whose orbits are well-defined and contained in $A$
\end{longtable}
\vspace{-1em}
\end{notation}

\subsubsection*{\textbf{\textup{Acknowledgements}}}
We would like to thank KAIST university and Dr. Nguyen Ngoc Cuong for hospitality and Professor Keiji Oguiso for the valuable discussions on the Fujiki varieties. 
The first author is supported by the NUS grant A-0004285-00-00 and the MOE grant
MOE-T2EP20120-0010. 
The second author is supported by the Institute for Basic Science (IBS-R032-D1-2023-a00).

\section{Preliminary}\label{s:pre}
In this section, we prepare some basic notions and terminologies to be used in the proofs of Theorems \ref{t:main1} and \ref{t:main2}. 
For the basic notion and terminology of complex spaces, we refer the reader to \cite{AG06} and \cite{Uen75}. 
The term \textit{(sub-)variety} in this paper will always denote a compact irreducible and reduced complex (sub-)space. 
A holomorphic map \(\pi\colon X\to Y\) of complex spaces is called a \textit{modification} if (1) \(\pi\) is proper and surjective, and (2) there exist nowhere Zariski dense analytic subsets \(M\) of \(X\) and \(N\) of \(Y\) such that \(\pi\) induces a biholomorphic map from \(X\setminus M\) onto \(Y\setminus N\). 
A \textit{resolution} of singularities of a variety \(X\) is a modification \(\widetilde X\to X\) from a nonsingular variety \(\widetilde X\).

\begin{definition}[Blowup]
Let \(X\) be a complex space and \(\mathcal{I}\subseteq\mathcal{O}_X\) an ideal sheaf of finite type, and \(Y\) the subspace of \(X\) defined by \(\mathcal{I}\).
The blowup of \(X\) centered at \(\mathcal{I}\) is a holomorphic map \(\pi\colon \widehat{X}\to X\) from a complex space \(\widehat{X}\) such that the following hold.
\begin{enumerate}
\item The inverse image of the ideal \(\pi^*(\mathcal{I})\cdot\mathcal{O}_{\widehat{X}}\subseteq\mathcal{O}_{\widehat{X}}\) generated by the functions \(\{s\circ \pi,s\in\mathcal{I}\}\) is an invertible sheaf, and hence the subspace \(\widehat{Y}\coloneqq\pi^{-1}(Y)\) of \(\widehat{X}\) defined by \(\pi^*(\mathcal{I})\cdot\mathcal{O}_{\widehat{X}}\) is a (not necessarily irreducible) hypersurface. 
Moreover, \(\pi\) induces an isomorphism \(\widehat{X}\setminus\widehat{Y}\cong X\setminus Y\) and thus is a modification.
\item If \(\varphi\colon X'\to X\) is another holomorphic map of complex spaces such that the ideal \(\varphi^*(\mathcal{I})\cdot\mathcal{O}_{X'}\) is invertible, then there exists a unique holomorphic map \(\tau\colon X'\to \widehat{X}\) such that \(\pi\circ\tau=\varphi\). 
In particular, the blowup, if it exists, is unique. 
\end{enumerate}
When \(\mathcal{I}=\mathcal{I}_Y\) is the ideal of the holomorphic functions on \(X\) vanishing along \(Y\), we say also that \(\pi\) is the blowup of \(X\) along a subspace \(Y\).
In this case, for a subspace \(Z\subseteq X\), we denote by \(\widehat{Z}\) the closure of \(\pi^{-1}(Z\setminus Y)\) in \(\widehat{X}\), which is called the \textit{strict transform} of \(Z\) through \(\pi\). 
\end{definition}
By Hironaka Chow lemma (see \cite[Corollary 2]{Hir75}), if \(\pi\colon X\to Y\) is a modification of compact complex spaces, then there exists a commutative diagram
\[
\xymatrix{
&\widehat{X}\ar[dl]_f\ar[dr]^{g}&\\
X\ar[rr]^\pi&&Y
}
\]
where \(f\) and \(g\) are obtained by finite sequences of blowup along smooth centers.

Now we recall the following equivalent definitions of compact Fujiki varieties. 
\begin{definition}[Fujiki variety]\label{d:fujiki}
Let \(X\) be a compact complex variety.
We say that \(X\) is of \textit{Fujiki class} \(\mathcal{C}\) (or \textit{Fujiki variety} for short) if one of the following equivalent conditions holds.
\begin{enumerate}
\item There is a holomorphic surjective map \(\widetilde{X}\to X\)
from a compact K\"ahler manifold $\widetilde X$ (thus, \(\dim(\widetilde{X})\geq\dim(X)\)); see \cite[Definition 1.1]{Fujiki-invent} and \cite[Lemma 4.6]{Fujiki-prims}. 
When $X$ is smooth, this is the original definition of \textit{Fujiki manifolds}. 
\item There is a holomorphic and bimeromorphic map \(\widetilde{X}\to X\) from a compact K\"ahler manifold $\widetilde X$; see \cite[Theorem 5]{Varouchas-mathann}. 
\end{enumerate}
\end{definition}
We note that \((1)\Rightarrow (2)\) is due to Hironaka's flattening theorem (see~\cite[Theorem 5 and its proof]{Varouchas-mathann} and \cite{Hir75}). 
By Blanchard's theorem \cite{Blanchard-ens}, a blowup of a compact K\"ahler manifold along a smooth submanifold is also a compact K\"ahler manifold (cf.~\cite[Proposition 3.24]{Voi02}); it follows that a subvariety of a Fujiki variety is also Fujiki (cf.~\cite[Lemma 4.6]{Fujiki-prims}). 

Next, we recall the notion of singular K\"ahler spaces, which was first introduced by Grauert \cite[Page  346]{Gra62} and was later modified by Moishezon \cite{Moi75}. 
We refer the reader to \cite[Chapter II]{Varouchas-mathann}, \cite[Definition 1.2]{Fujiki-prims}, \cite[Sections 2 and 3]{HP16}, \cite[Section 3]{GK-AIF}, and the references therein for more information.
\begin{definition}[Singular K\"ahler space]\label{d:cpt}
Let \(X\) be an irreducible and reduced complex space. 
Let \(\mathscr{C}_{X,\mathbb{R}}^{\infty}\) be the sheaf of smooth real-valued functions on \(X\), and \(\textup{PH}_{X,\mathbb{R}}\) the image of the real part map \(\mathcal{O}_X\to \mathscr{C}_{X,\mathbb{R}}^{\infty}\) which we call the sheaf of real-valued pluriharmonic functions on \(X\). 
Set \(\mathscr{K}_{X,\mathbb{R}}^1\coloneqq\mathscr{C}_{X,\mathbb{R}}^{\infty}/\textup{PH}_{X,\mathbb{R}}\).

An element \(\kappa\) in \(H^0(X,\mathscr{K}_{X,\mathbb{R}}^1)\) is called a \textit{K\"ahler metric} on \(X\) if 
it can be represented by a family \(\{(U_j,\varphi_j)\}_{j\in J}\) for some open covering \(\{U_j\}_{j\in J}\) of \(X\) satisfying the following conditions.
\begin{enumerate}
\item The function \(\varphi_j\colon U_j\to\mathbb{R}\) is a smooth strictly plurisubharmonic (psh for short) \(\mathscr{C}^{\infty}\)-function on \(U_j\) for all \(j\in J\), i.e., \(\varphi_j\) is locally induced by a smooth strictly psh \(\mathscr{C}^{\infty}\)-function on an open subset of \(\mathbb{C}^{N_j}\) via a local embedding \(U_j\hookrightarrow\mathbb{C}^{N_j}\); and
\item  On \(U_{ij}\coloneqq U_i\cap U_j\), the function \(\varphi_i|_{U_{ij}}-\varphi_j|_{U_{ij}}\) is pluriharmonic. 
\end{enumerate}
We say that \(X\) is \textit{K\"ahler} if there exists a K\"ahler metric on \(X\).
An element \(c\in H^1(X,\textup{PH}_{X,\mathbb{R}})\) is called a \textit{K\"ahler class} on \(X\) if there exists a K\"ahler metric \(\kappa\) on \(X\) such that \(c\) is the image of \(\kappa\) via the natural map
\[
H^0(X,\mathscr{K}_{X,\mathbb{R}}^1)\xrightarrow{P_0} H^1(X,\textup{PH}_{X,\mathbb{R}})
\]
which is the connecting homomorphism in degree 0 associated with the short exact sequence 
\[0\to \textup{PH}_{X,\mathbb{R}}\to \mathscr{C}_{X,\mathbb{R}}^{\infty}\to\mathscr{K}_{X,\mathbb{R}}^1\to 0.\] 
By \cite[Remark 3.2]{GK-AIF}, as \(\mathscr{C}_{X,\mathbb{R}}^{\infty}\) is a fine sheaf, the map \(P_0\) is always surjective.
Besides, there is also a natural map
\[
H^1(X,\textup{PH}_{X,\mathbb{R}})\xrightarrow{P_1} H^2(X,\mathbb{R})
\]
which is the connecting homomorphism in degree 1 associated with the short exact sequence 
\[0\to \mathbb{R}_X\xrightarrow{i\cdot}\mathcal{O}_X\xrightarrow{\textup{Re}}\textup{PH}_{X,\mathbb{R}}\to 0.\]
By \cite[Remark 3.7]{HP16}, if \(X\) has at worst rational singularities, then \(P_1\) is injective (cf.~\cite[Proposition 3.5]{GK-AIF}), in which case, one can define the intersection product on \(H^1(X,\textup{PH}_{X,\mathbb{R}})\) via the cup product. 
The cohomology group \(H^1(X,\textup{PH}_{X,\mathbb{R}})\) is also known as the \textit{Bott-Chern cohomology} (see \cite[Definition 3.1]{HP16}; cf.~\cite[Definition 4.6.2]{BG13}). 
If \(X\) is a compact K\"ahler manifold, then from the Hodge theory, there is a natural isomorphism  between the Bott-Chern cohomology group and the Hodge group
with real coefficients  \(H^1(X,\textup{PH}_{X,\mathbb{R}})\cong H^{1,1}(X)\cap H^2(X,\mathbb{R})\). 
\end{definition}

\begin{definition}[{Differential form; see \cite[Section 1]{Dem85}}]\label{d:differential}
Let \(X\) be an analytic variety.
A differential form \(\omega\) of type \((p,q)\) on \(X\) is a differential form (of class \(\mathscr{C}^{\infty}\)) of type \((p,q)\) on the smooth locus \(X_{\reg}\) such that for every point \(x\in X_{\sing}\) in the singular locus, there is an open neighborhood \(x\in U\subseteq X\) and a closed embedding \(U\hookrightarrow V\) into an open set \(V\subseteq \mathbb{C}^N\) such that there exists a differential form (of class \(\mathscr{C}^{\infty}\)) \(\omega_V\) of type \((p,q)\) on \(V\) such that \(\omega_V|_{U\cap X_\reg}=\omega|_{U\cap X_\reg}\). 
Denote the sheaf of \((p,q)\)-forms by \(\mathcal{A}_X^{p,q}\) and denote by \(\mathcal{A}_{\mathbb{R},X}^{p,p}\) the sheaf of real forms of bidegree \((p,p)\).
Note that we only consider \(\mathscr{C}^{\infty}\) forms here.

A \textit{K\"ahler form} on \(X\) is a positive closed real \((1,1)\)-form \(\omega\in\mathcal{A}_{\mathbb{R},X}^{1,1}(X)\) such that for every point \(x\in X_\sing\) there is an open neighborhood \(x\in U\subseteq X\), a closed embedding \(U\hookrightarrow  V\) into an open subset \(V\subseteq\mathbb{C}^N\), and a strictly psh \(\mathscr{C}^{\infty}\)-function \(f\colon V\to\mathbb{R}\) such that \(\omega|_{U\cap X_\reg}=(i\partial\overline{\partial}f)|_{U\cap X_\reg}\).
We observe that, by applying the real operator \(i\partial\overline{\partial}\), if an analytic variety \((X,\kappa)\) is K\"ahler equipped with a K\"ahler metric \(\kappa=\{(U_j,\varphi_j)\}\) in the sense of Definition \ref{d:cpt},
we obtain a global section \(\mathcal{A}_{\mathbb{R},X}^{1,1}(X)\), that is a smooth K\"ahler form \(\omega\coloneqq\{(U_j,i\partial\overline{\partial}\varphi_j)\}\) (cf.~\cite[Page 23, \S 1.1]{Varouchas-mathann}).
However, a K\"ahler form does not necessarily determine a K\"ahler metric unless \(X\) is smooth. 

An element \(\kappa\in H^0(X,\mathscr{K}^{1}_{X,\mathbb{R}})\) is also known as a closed \((1,1)\)-form via the real operator \(i\partial\overline{\partial}\). 
If \(\omega_1\) and \(\omega_2\) are two  closed real \((1,1)\)-forms (which are induced from \(i\partial\overline{\partial}\))  
such that \(P_0(\omega_1)=P_0(\omega_2)\), then we have \(\omega_1=\omega_2+i\partial\overline{\partial}u\) where \(u\in\mathscr{C}_{X,\mathbb{R}}^{\infty}\) is a globally smooth real-valued function (cf.~\cite[Remark 3.2]{HP16}); in this case, we denote by \([\omega_1]=[\omega_2]\in H^1(X,\textup{PH}_{X,\mathbb{R}})\).

Let \(f\colon X\to Y\) be a proper holomorphic map between compact complex spaces. 
By \cite[Lemma 1.3]{Dem85}, we can define the pullback of differential forms by gluing the local definitions together so as to obtain a pullback map:
\[
f^*\colon \mathcal{A}^{p,q}(Y)\to\mathcal{A}^{p,q}(X).
\]
\end{definition}

\begin{remark}\label{r:2cpt}
In \cite[\S 1.1, Page 384]{Zhang-distribution}, a \textit{K\"ahler variety} \(X\) with a \textit{K\"ahler form} \(\omega\) is defined as an analytic variety which admits a finite map \(f\colon X\to M\) to a K\"ahler manifold \(M\) such that \(\omega=f^*\eta\) for some K\"ahler form \(\eta\) on \(M\). 
In view of \cite[Proposition 3.6]{GK-AIF} (cf.~\cite{Vaj96}), if a compact analytic variety is a K\"ahler variety in the sense of \cite{Zhang-distribution}, it is K\"ahler in the sense of Definition \ref{d:cpt}. 
Besides, \cite[Conjecture 1.2.2]{Zhang-distribution} asserts the asymptotic growth for periodic points in any periodic subvarieties. 
As every subvariety of a K\"ahler space is also K\"ahler (see \cite[1.3.1]{Varouchas-mathann} or \cite[Proposition 3.6]{GK-AIF}) and the restriction of any iterate of a polarized self-map to an analytic invariant subvariety is also polarized, our Corollary \ref{c:swz} gives a positive answer to \cite[Conjecture 1.2.2]{Zhang-distribution}.
\end{remark}

In the following, we stick to Demailly's book \cite{Demailly-book}  for the basic notions involved. 
Now, we recall the pullback operator on currents and introduce a notion called \textit{nice current} which will be used in the proof of Proposition \ref{p:small-dyn}. 

Let $f\colon X\to Y$ be a meromorphic map between compact complex manifolds, i.e., \(f\) is a holomorphic map defined on a dense (analytic) Zariski open subset of \(X\) such that the closure \(\Gamma\) of its graph in $X\times Y$ is an irreducible analytic set of dimension equal to  $\dim (X)$. 
Here, we don't assume that $f$ is dominant or generic fibers are finite. 
Denote by $\pi_X$ and $\pi_Y$ the projections from $X\times Y$ to its factors respectively. 
The \textit{indeterminacy set} of \(f\) consists of the point $x\in X$ such that the set $\pi_X^{-1}(x)\cap\Gamma$ is not a singleton or equivalently of positive dimension.  
The map \(f\) induces linear operators on forms and currents.
If $T$ is a current on \(Y\), then its pullback $f^*(T)$ on \(X\) is formally defined by
$$f^*(T)\coloneqq(\pi_X)_*(\pi_Y^*(T)\wedge [\Gamma]),$$
when the last wedge-product is well-defined (cf.~\cite{DS07}).  

If $T$ is a smooth form, then $f^*(T)$ is well-defined and is a smooth form outside the indeterminacy locus of $f$; 
this form is locally integrable of class $L^1$ and then it defines a current of order 0 without mass on any proper analytic subsets of $X$. 
In particular, to define $f^*(T)$ for \(T\) smooth, it is enough to define it in any dense Zariski open subset. 
As an extension of this special case, we introduce the following notion.

\begin{definition}[Nice currents]
Let \(X\) be a compact K\"ahler manifold. 
A positive closed current \(T\) on \(X\) is called \textit{nice} if \(T\) has no mass on proper analytic subsets of \(X\) and the restriction of \(T\) to some dense Zariski open set is given by a continuous differential form. 
For a nice positive closed current $\Omega$ of maximal degree on $X$, we define its mass
$$\int_X^\bullet \Omega$$
as the integral of $\Omega$ on a Zariski open set where it is continuous. 
As a nice current has no mass on any proper analytic subset, the mass of \(\Omega\) is independent of the choice of the Zariski open set. 
\end{definition}

We have the following pullback of nice currents along a dominant meromorphic map.

\begin{lemma}\label{l:nice-pullback}
Let $\pi\colon X'\to X$ be a dominant meromorphic map between compact K\"ahler manifolds, not necessarily of the same dimension. 
Let $T$ be a nice positive closed current on $X$. 
Let $U$ be a Zariski open subset of $X'$ such that $\pi$ is holomorphic on $U$ and $T$ is given by a continuous differential form on $\pi(U)$. 
Then the form $\pi^*(T)$ on $U$ defines a nice positive closed current on $X'$ which is independent of the choice of $U$. We will denote it by $\pi^\bullet(T)$. 
\end{lemma}
\begin{proof}
The independence of $U$ is clear as nice positive closed currents have no mass on proper analytic subsets. 
Therefore, if $\pi^\bullet(T)$ exists, it has to be the extension by zero of the current $\pi^*(T)$ on $U$. 
It is enough to show that the mass of $\pi^*(T)$ is finite. Indeed, by Skoda's theorem \cite{Sko82}, the finite mass allows us to extend $\pi^*(T)$ by zero to a positive closed current on $X'$. 

By the regularization theorem \cite[Corollary 1.2 and Proposition 4.6]{DS-ens}, 
there is a sequence of smooth positive closed forms $T_n$ on $X$ satisfying the following: (1) the mass  $\|T_n\|$ is bounded by a constant which is independent of $n$, (2) $T_n$ converges weakly to a positive closed current $T'\geq T$, and (3) the convergence is locally uniform on $\pi(U)$. 

Recall that the mass of a positive closed current on a compact K\"ahler manifold depends only on its cohomology class. 
Observe that the cohomology class of $T_n$ is bounded independently of $n$. 
It follows that the cohomology class of $S_n\coloneqq\pi^*(T_n)$ is also bounded. Hence the mass $\|S_n\|$ (which is well-defined as $T_n$ is smooth) is bounded by a constant. So this sequence is relatively compact and by taking a subsequence we can assume that \(S_n\) converges to some positive closed current $S$ on $X'$. 
As $S\geq \pi^*(T)$ on $U$ by using the uniform convergence of $T_n$ and $S$ is defined on the whole $X'$, we deduce that \(\pi^*(T)\) has finite mass. 
\end{proof}

\begin{remark}\label{r:nice-pullback-Fujiki}
Let $\pi'\colon X''\to X'$ and \(\pi\colon X'\to X\) be  dominant meromorphic maps between compact K\"ahler manifolds. Then for a nice current \(T\) on \(X\), we have 
$$(\pi\circ\pi')^\bullet (T) = (\pi')^\bullet \pi^\bullet (T).$$
\end{remark}

To end this preliminary section, we review the dynamical degrees of a dominant meromorphic map, which are the key invariants in the dynamical systems. 
\begin{definition}[Dynamical degrees]\label{d:dynamical-degree}
Let \(f\colon X\to X\) be a dominant meromorphic map on a compact K\"ahler manifold of dimension \(k\), and \(\omega\) a normalized K\"ahler form so that \(\omega^k\) defines a probability measure on \(X\). 
For each \(p=0,\cdots,k\), the \textit{\(p\)-th dynamical degree} of \(f\) is given by
\[
d_p(f)\coloneqq\lim_{n\to\infty}\|(f^n)^*(\omega^p)\|^{1/n}=\lim_{n\to\infty}\|(f^n)^*\colon H^{p,p}(X,\mathbb{C})\to H^{p,p}(X,\mathbb{C})\|^{1/n}.
\]
Here, if \(T\) is a positive closed \((p,p)\) current on \(X\), then its mass is given by the formula \(\|T\|\coloneqq\left<T,\omega^{k-p}\right>\). 
We also use a fixed norm \(\|\cdot\|\) on \(H^{p,p}(X,\mathbb{C})\). 
The \textit{algebraic entropy} of \(f\) is defined by
\[
h_a(f)\coloneqq\max_{0\leq p\leq k}\text{log}\,d_p(f).
\]
It is known that the above limits exist and the topological entropy of \(f\) is always bounded by the algebraic entropy; see \cite{DNT12, DS-ens, Dinh-Sibony-Annals05, Gro77, Yom87} and the references therein.
Note also that the dynamical degrees are bimeromorphic invariants (see \cite{DS-ens}, \cite{Dinh-Sibony-Annals05} for details; cf.~\cite{DNT12,Tru15,Tru20} for some extensions of this result). 
Due to a mixed version of the Hodge-Riemann theorem \cite{DN-GAFA06, Gromov-90, Khovanskii-79, Teissier79, Timorin-1998},  the function \(p\mapsto \text{log}\,d_p(f)\) is concave, that is, \(d_p(f)^2\ge d_{p-1}(f)d_{p+1}(f)\). 

We say that \(f\) has a \textit{dominant topological degree}, if the last dynamical degree (also called the \textit{topological degree})  \(d_t\coloneqq d_k(f)\) is strictly larger than the other dynamical degrees (cf.~Remark \ref{r:proj:dominant} and Proposition \ref{p:dominant-int}).
By the log-concavity of dynamical degrees, this is equivalent to \(d_t>d_{k-1}(f)\). 
It is known that for a dominant meromorphic map with a dominant topological degree, the following weak limit of probability measures
\[
\mu\coloneqq\lim_{n\to\infty}\frac{1}{d_t^n}(f^n)^*\omega^k
\]
exists.
The probability measure \(\mu\) is called the \textit{equilibrium measure} of \(f\) which has no mass on any proper analytic subset of \(X\). 
We refer to \cite{DS-ens,DS06} for details.

Now, let \(f\colon X\to X\) be a dominant meromorphic map on a Fujiki variety \(X\). 
By Definition \ref{d:fujiki}, there is a blowup \(\pi\colon\widetilde{X}\to X\) from a compact K\"ahler manifold \(\widetilde{X}\) together with an induced dominant meromorphic map \(\widetilde{f}\) on \(\widetilde{X}\) such that \(\pi\circ\widetilde f=f\circ\pi\). 
Then we define the \(p\)-th dynamical degree \(d_p(f)\) as \(d_p(\widetilde{f})\), which is independent of the choice of the smooth K\"ahler model \(\widetilde{X}\). 
\end{definition}

When \(f\colon X\to X\) is a holomorphic self-map of a compact K\"ahler manifold \(X\) with a dominant topological degree, we extend \cite[Proposition 3.7]{MZ23} to the analytic case. 
\begin{proposition}\label{p:dominant-int}
Let \(f\colon X\to X\) be a surjective holomorphic self-map of a compact K\"ahler manifold of dimension \(k\).
Then \(f\) has a dominant topological degree if and only if \(f\) is int-amplified.
\end{proposition}
\begin{proof}
Recall that \(f\) is called int-amplified if all the eigenvalues of \(f^*|_{H^{1,1}(X,\mathbb{R})}\) are of modulus greater than 1 (see \cite[Theorem 1.1]{Zhong-Asian}).
Besides, if \(f\) is int-amplified, then it follows from \cite[Lemma 3.4]{Zhong-Asian} (cf.~\cite[Lemma 3.6]{Men20}) that \(f\) has a dominant topological degree. 
Hence, we only need to show the other direction by assuming in the following that \(f\) has a dominant topological degree. 
By the log-concavity of dynamical degrees (see Definition \ref{d:dynamical-degree}), we have 
\[d_t(f)\coloneqq d_k(f)>d_{k-1}(f)>\cdots>d_0(f).\] 
We note that \(d_{k-1}(f)\) is the spectral radius of \(f^*|_{H^{k-1,k-1}(X,\mathbb{R})}\) (see \cite[Proposition 5.8]{Din05}).
As \(f\) is surjective and hence finite (cf.~e.g.~\cite[Lemma 2.10]{Zhong-Asian}), it follows from \cite{DS07} that we can pull back \(i\partial\overline{\partial}\)-closed currents of bidegree \((k-1,k-1)\) and it is compatible with the cohomology class.
From the projection formula, for any \(\alpha\in H^{k-1,k-1}(X,\mathbb{R})\), we have 
\(f_*f^*\alpha=\deg(f)\cdot\alpha\).
In particular,
\[
f_*|_{H^{k-1,k-1}(X,\mathbb{R})}=\deg(f)\cdot (f^*|_{H^{k-1,k-1}(X,\mathbb{R})})^{-1}.
\]
As \(d_t(f)=\deg(f)>d_{k-1}\), it follows that all of the eigenvalues of \(f_*|_{H^{k-1,k-1}(X,\mathbb{R})}\) are of modulus greater than 1.
By the Poincar\'e duality, we obtain that
\[
f^*|_{H^{1,1}(X,\mathbb{R})}=(f_*|_{H^{k-1,k-1}(X,\mathbb{R})})^{\vee}
\]
and hence all of the eigenvalues of \(f^*|_{H^{1,1}(X,\mathbb{R})}\) are of modulus greater than 1.
\end{proof}

\begin{remark}\label{rmk:singular:dominant-int}
In the above proposition, when \(X\) is a normal compact K\"ahler space with at worst rational singularities, we still have a perfect pairing
\[
N^1(X)\coloneqq H^1(X,\textup{PH}_{X,\mathbb{R}}) \times N_1(X)\to\mathbb{R},
\]
where \(H^1(X,\textup{PH}_{X,\mathbb{R}})\) is defined as in Definition \ref{d:cpt} (see~\cite[Proposition 3.9]{HP16}) and \(N_1(X)\) is the vector space of real closed currents of bidimension \((1,1)\) modulo the equivalence condition: \(T_1\equiv T_2\) if and only if \(T_1(\eta)=T_2(\eta)\) for all real closed \((1,1)\)-forms \(\eta\) with local potentials.
However, we don't know in the singular setting whether there is an effective way to define \(f^*|_{N_1(X)}\) like \cite{DS07}.
\end{remark}

\section{Proof of Theorem \ref{t:main2} and further remarks}\label{s:main2}
In this section, after preparing a comparison proposition between the dynamical degrees of a dominant meromorphic map on the ambient space and those of their restrictions to invariant proper analytic subsets, we prove Theorem \ref{t:main2}. 

First of all, let us recall the following theorem which is proved in \cite{DNT-BLMS}.
We shall reduce the proof of our Theorem \ref{t:main2} to the case of compact K\"ahler manifolds, i.e., Theorem \ref{t:per-kahler} below.

\begin{theorem}[{see \cite[Theorem 1.1]{DNT-BLMS}}]\label{t:per-kahler}
Let $f\colon X\to X$ be a dominant meromorphic map on a compact K\"ahler manifold and $h_a(f)$ its algebraic entropy.
Denote by \(\widetilde \IP_n(f)\) the set of isolated \(f\)-periodic points of period \(n\) counted with multiplicity. 
Then we have
\[
\limsup_{n\to\infty}\frac{1}{n}\textup{log}\,\#\widetilde \IP_n(f)\leq h_a(f).
\]
\end{theorem}

As a consequence of Theorem \ref{t:per-kahler}, we obtain the following lemma which is a weak version of Theorem \ref{t:main2}.
\begin{lemma}\label{l:singular-periodic}
Let $f\colon X\to X$ be a dominant meromorphic map on a Fujiki variety and $h_a(f)$ its algebraic entropy. 
Denote by \(\IP_n^0(X\backslash I_f,f)\) the set of isolated \(f\)-periodic points of period \(n\) whose orbit do not intersect the indeterminacy set \(I_f\). 
Then there is a dense Zariski open subset \(U\subseteq X\) such that 
\[
\limsup_{n\to\infty}\frac{1}{n}\textup{log}\,\#(\IP_n^0(X\backslash I_f,f)\cap U)\leq h_a(f).
\]
\end{lemma}
\begin{proof}
Let \(\pi\colon\widetilde{X}\to X\) be a blowup such that \(\widetilde{X}\) is a compact K\"ahler manifold and let $\widetilde f$ be the dominant meromorphic map on \(\widetilde X\) induced by \(f\). 
Then there is a dense Zariski open subset \(U\subseteq X\) such that \(\pi^{-1}(U)\) is biholomorphic to \(U\).
Hence, the lemma follows from the following inequality
\[
\#(\IP_n^0(X\backslash I_f,f)\cap U)\leq \# \widetilde \IP_n(\widetilde{f})
\]
and Theorem \ref{t:per-kahler} by noting that \(h_a(f)=h_a(\widetilde{f})\).
\end{proof}

The following comparison proposition is known for a holomorphic surjective self-map of a compact K\"ahler manifold (see \cite[Appendix A, Proposition A.11]{NZ09}). 
We extend it to the case of a meromorphic map of a Fujiki variety and this will be crucially used in the proofs of our main results.
Note that it also has its own independent interests (cf.~Remark \ref{r:no product formula}). 
\begin{proposition}\label{p:small-dyn}
Let \(f\colon X\to X\) be a dominant meromorphic map of a Fujiki variety \(X\) of dimension \(k\), and \(Z\) a proper \(f\)-invariant irreducible analytic subset which is not contained in the indeterminacy set \(I\) of \(f\), i.e., the closure of \(f(Z\setminus I)\) is equal to \(Z\). 
Then the \(p\)-th dynamical degree of $f_{|Z}$ is no more than $d_p(f)$. 
In particular, the algebraic entropy \(h_a (f_{|Z})\) is no more than \(h_a(f)\).
Moreover, if \(f\) has a dominant topological degree, then
all dynamical degrees of $f_{|Z}$ are no more than $d_{k-1}(f)$.
\end{proposition}

\begin{proof}
Note that the last conclusion is due to the log-concavity of the dynamical degrees (see Definition \ref{d:dynamical-degree}).  
Consequently, we have 
\[d_k(f)>d_{k-1}(f)>\cdots>d_0(f)\]  by the further assumption of the dominant topological degree of \(f\), and thus we conclude by noting that \(\dim(Z)\leq k-1\).

In the following, we shall verify that \(d_p(f_{|Z})\leq d_p(f)\) for each \(0\leq p\leq \dim(Z)\).
Let \(\pi\colon Z'\to Z\) be a holomorphic and bimeromorphic map from a compact K\"ahler manifold, noting that the analytic subvariety \(Z\) is also in the Fujiki class \(\mathcal{C}\) (see \cite[Lemma 4.6]{Fujiki-prims}). 
Then the algebraic entropy of \(f_{|Z}\) is equal to the one of the composite map
\[g\coloneqq\pi^{-1}\circ f_{|Z}\circ \pi\colon Z'\to Z'.\]

Let $\Pi\colon \widetilde X \to X$ be a bimeromoprhic and holomorphic map from a compact K\"ahler manifold \(\widetilde X\).
Note that $\Pi^{-1}(Z)$ may have some component of dimension larger than $\dim (Z)$. 
Note also that the induced dominant meromorphic map 
\[\widetilde f\coloneqq\Pi^{-1}\circ f\circ \Pi\colon \widetilde X\to\widetilde X
\] may send a certain component of $\Pi^{-1}(Z)$ to an analytic subset of smaller dimension; see also Remark \ref{r:no product formula}.  
With \(\widetilde{X}\) further replaced by a log resolution with respect to the inverse image of the ideal sheaf \(\Pi^*\mathcal{I}_Z\) due to Hironaka  (cf.~\cite{Wlo09} and \cite[Corollary 2]{Hir75}), 
we may assume the following:
\begin{enumerate}
\item The model \(\widetilde{X}\) is a compact K\"ahler manifold.
\item The support of the inverse image \(\text{Supp}\,\Pi^{-1}(Z)\) is a simple normal crossing divisor, and hence every irreducible component is a smooth hypersurface. 
\item The map \(f\) lifts to a dominant meromorphic map \(\widetilde{f}\) on \(\widetilde{X}\).
\end{enumerate}
Let \(\widetilde Z_0,\cdots,\widetilde Z_m\) be all of the irreducible components of \(\Pi^{-1}(Z)\), any of which dominates \(Z\). 
Let us fix a component $\widetilde Z\coloneqq\widetilde Z_0$ (such that $\Pi(\widetilde Z)=Z$).   
Since $\widetilde Z$ is an hypersurface, it is not contained in the indeterminacy set \(\widetilde I_n\) of $\widetilde f^n$ for every \(n\) 
and then we have 
\[\widetilde f^n(\widetilde Z\setminus \widetilde I_n)\subseteq \widetilde Z_{i_n}\]
where \(0\leq i_n\leq m\) for each \(n\).
Here, we need that \((f^n)_{|Z}\) is dominant for each \(n\geq 1\), though \((\widetilde f^n)|_{\widetilde Z\setminus \widetilde I_n}\) is not necessarily dominant over \(\widetilde Z_{i_n}\). 
Note also that  
the choice of \(i_n\) is not necessarily unique. 
Then the surjective map $\Pi$ induces the following dominant meromorphic maps  
\[
\tau_n\coloneqq\pi^{-1}\circ\Pi_{|\widetilde Z_{i_n}}\colon\widetilde Z_{i_n}\to Z'
\]
such that \(g^n\circ\tau_0=\tau_n\circ (\widetilde f^n)_{|\widetilde Z}\). 

We aim to make each \(\tau_n\) holomorphic. 
Replacing $\widetilde X$ with a suitable blowup with smooth center in \(\widetilde Z_0\), we may assume that $\tau_0$ is holomorphic (cf.~\cite[Corollary 2 and Definition 4.1]{Hir75}). 
In fact, the indeterminacy set of \(\tau_0\colon\widetilde Z_0\to Z'\) is contained in the inverse image \((\Pi_{|\widetilde Z_0})^{-1}(I_{\pi^{-1}})\), where \(I_{\pi^{-1}}\) is a proper analytic subset consisting of the singular locus and the ``non-K\"ahler'' locus of \(Z\), over which \(\pi\) is not isomorphic.  
Besides, a resolution \(\overline{\Gamma_0}\) of singularities of the  graph \(\Gamma_0\) of \(\tau_0\) only has positive dimensional fibre over \((\Pi_{|\widetilde Z_0})^{-1}(I_{\pi^{-1}})\). 
Applying Hironaka Chow lemma for an open neighborhood of \((\Pi_{|\widetilde Z_0})^{-1}(I_{\pi^{-1}})\) (see \cite[Corollary 2]{Hir75}), we obtain that there is a blowup along smooth centers of \(\widetilde{Z_0}\) (contained in the indeterminacy set in \((\Pi_{|\widetilde Z_0})^{-1}(I_{\pi^{-1}})\)) which dominates \(\overline{\Gamma_0}\).
From this construction (see~\cite[pp. 137-138]{Moi74} and \cite[(4.2.1)]{Hir75}), the blowup does not create any new component of \(\Pi^{-1}(Z)\) dominating \(Z'\); for each \(0\leq i\leq m\), we can identify \(\widetilde{Z}_i\) with its strict transform.  
Inductively, after finitely many blowups along smooth centers in \(\Pi^{-1}(Z)\), we may assume that each \(\tau_n\) is holomorphic. 
For the convenience of the reader, we summarize the above construction via the following commutative diagram.
\[
\xymatrix{
\widetilde{Z}\ar@/^2pc/[rrr]^{(\widetilde f^n)_{|\widetilde Z}}\ar@{^(->}[r]\ar[dd]_{\tau_0}&\widetilde{X}\ar[r]^{\widetilde f^n}\ar[d]_{\Pi}&\widetilde{X}\ar[d]^{\Pi}&\widetilde Z_n\ar[dd]^{\tau_n}\ar@{_(->}[l]\\
&X\ar[r]^{f^n}&X&\\
Z'\ar@/_2pc/[rrr]_{g^n}\ar[r]^{\pi}&Z\ar[r]^{(f^n)_{|Z}}\ar@{^(->}[u]&Z\ar@{^(->}[u]&Z'\ar[l]_{\pi}
}
\]

Denote by $k$ the dimension of $X$ and $l$ the one of $Z$. 
The dimension of $\widetilde Z$ is then $k-1$.  
Fix a K\"ahler form \(\alpha\) on \(Z'\) and a K\"ahler form $\widetilde\omega$ on $\widetilde{X}$. 
Then $(\tau_0)_*(\widetilde\omega_{|\widetilde Z})^{k-1-l}$ (which is the integration along the fibres of the form \((\widetilde\omega_{|\widetilde Z})^{k-1-l}\)) is a positive closed $(0,0)$-current on $Z'$ and hence is given by a positive constant function. 
Multiplying $\widetilde\omega$ by a constant, we may  assume that $(\tau_0)_*(\widetilde\omega_{|\widetilde Z})^{k-1-l}=[Z']$.

Since $g^n\circ \tau_0 = \tau_n\circ (\widetilde f^n)_{|\widetilde Z}$ for each \(n\) and any smooth \((p,p)\)-form on \(\widetilde Z_{i_n}\) is bounded by a constant times \((\widetilde \omega^p)_{|\widetilde Z_{i_n}}\), applying Lemma \ref{l:nice-pullback} and Remark \ref{r:nice-pullback-Fujiki} and noting that \(\tau_n\) is dominant, we obtain the following
\[\tau_0^\bullet(g^n)^*(\alpha^p)=[(\widetilde f^n)_{|\widetilde Z}]^\bullet \tau_n^*(\alpha^p)\lesssim [(\widetilde f^n)_{|\widetilde Z}]^\bullet((\widetilde \omega^p)_{|\widetilde Z_{i_n}}).
\]
As there are only finitely many choices of \(i_n\), we can take a uniform constant \(C\) which is independent of \(n\) such that for any \(n\), 
\[
\tau_n^*(\alpha^p)\leq C\cdot (\widetilde \omega^p)_{|\widetilde Z_{i_n}}.
\]
Fix an integer \(0\leq p\leq l\). 
Then we have
\begin{align*}
\int_{Z'} (g^n)^*(\alpha^p)\wedge \alpha^{l-p}&=
\int_{\widetilde Z}^\bullet \tau_0^\bullet(g^n)^*(\alpha^p)\wedge \tau_0^*(\alpha^{l-p})\wedge (\widetilde\omega^{k-l-1})_{|\widetilde Z}\\
&=\int_{\widetilde Z}^\bullet [(\widetilde f^n)_{|\widetilde Z}]^\bullet\tau_n^*(\alpha^p) \wedge \tau_0^*(\alpha^{l-p})\wedge (\widetilde\omega^{k-l-1})_{|\widetilde Z}\\
&\lesssim \int_{\widetilde Z}^\bullet [(\widetilde f^n)_{|\widetilde Z}]^\bullet((\widetilde \omega^p)_{|\widetilde Z_{i_n}}) \wedge  (\widetilde\omega^{k-p-1})_{|\widetilde Z}
= \int_{\widetilde Z}^\bullet (\widetilde f^n)^*(\widetilde\omega^p) \wedge  \widetilde\omega^{k-p-1},
\end{align*}
where the last equality is due to Lemma \ref{l:nice-pullback}.
By \cite[Corollary 1.2 and Proposition 4.6]{DS-ens}, for each \(n\), there are smooth currents \(\{T_{n,s}\}_s\) converging to a current \(T_n\geq (\widetilde f^n)^*(\widetilde \omega^p)\) and the convergence is locally uniform outside the indeterminacy set \(\widetilde I_n\).
Moreover, the mass of \(T_{n,s}\) is bounded by a constant times the one of \((\widetilde f^n)^*(\widetilde \omega^p)\) and the constant is independent of \(n\). 
It then follows that 
\[
\int_{\widetilde Z}^\bullet (\widetilde f^n)^*(\widetilde\omega^p) \wedge  \widetilde\omega^{k-p-1}\leq\limsup_{s\to\infty}\int_{\widetilde Z} T_{n,s}\wedge \widetilde\omega^{k-p-1}.
\]
As the last integral depends only on the cohomology classes of \(T_{n,s}\) (which are uniformly bounded by a multiple of \((\widetilde f^n)^*(\widetilde\omega^p)\) and the multiple is independent of \(s\) and \(n\)), we have the following inequality
\[
\int_{Z'} (g^n)^*(\alpha^p)\wedge \alpha^{l-p}\lesssim \int_{\widetilde Z} (\widetilde f^n)^*(\widetilde\omega^p)\wedge \widetilde\omega^{k-p-1}\lesssim \|(\widetilde f^n)^*(\widetilde \omega^p)\|.
\]
By taking the limit of the \(n\)-th root, 
we deduce that the dynamical degree of order $p$ of $f_{|Z}$ is no more than the one of $\widetilde f$ which is equal to $d_p(f)$. The lemma then follows.
\end{proof}

\begin{remark}\label{r:no product formula}
When proving Proposition \ref{p:small-dyn}, as the induced map $\widetilde f\coloneqq\Pi^{-1}\circ f\circ \Pi$ may send certain component of $\Pi^{-1}(Z)$ to an analytic subset of smaller dimension, even if \(\widetilde f\) fixes some component \(\widetilde Z\) of \(\Pi^{-1}(Z)\), one could not simply apply the product formula (see \cite[Theorem 1.1]{DN11} and \cite[Theorem 1.1]{DNT12}) to conclude that \(h_a(f_{|Z})\leq h_a(\widetilde{f}_{|\widetilde{Z}})\). 
\end{remark}

\begin{proof}[Proof of Theorem \ref{t:main2}]
Denote by \(P_n\coloneqq\IP_n^0(X\setminus I_f,f)\) the set of all isolated \(f\)-periodic points of period \(n\) whose orbits do not intersect \(I_f\). 
Suppose to the contrary that there exists a subsequence \(\{n_i\}\) and a number \(\lambda>1\) such that for all \(i\in\mathbb{N}\), we have
\begin{align}\label{eq-geq1}
\# P_{n_i}\geq \exp(n_i\lambda h_a(f)).
\end{align}
By Lemma \ref{l:singular-periodic}, 
there is a dense Zariski open set \(U\subseteq X\) such that 
\[
\limsup_{n\to\infty}\frac{1}{n}\text{log}\,\#(P_n\cap U)\leq h_a(f).
\]
Therefore, there exists some \(\mu<\lambda\) such that for all sufficiently large \(n\), we have
\[
\#(P_n\cap U)\leq \exp(n\mu h_a(f)).
\]
This implies that there exists some positive integer \(\delta<1\) such that 
\begin{align}\label{eq-geq2}
\#(P_{n_i}\cap Z)\geq\delta\# P_{n_i}
\end{align}
for all sufficiently large \(n_i\), where \(Z\coloneqq X\setminus U\) is a proper analytic subset (for this purpose, we can take \(\delta=1/2\)). 
With \(Z\) replaced by a proper closed analytic subset, \(\{n_i\}\) replaced by a subsequence and \(\delta\) replaced by another positive number, we may assume that \(Z\) has the minimal dimension, 
in other words, we find a proper analytic subset \(Z\) of minimal dimension such that 
\[\# (P_{n_i}\cap Z)\geq \delta\# P_{n_i}.\]

Since \(P_n\) is invariant by $f$ and \(f\) is bijective from \(P_n\) to \(P_n\), we see that $\# (P_{n_i}\cap f^s(Z))\geq \delta (\# P_{n_i})$ for every $i$ and every \(s\). 
Here, we define \(f^s(Z)\) as the closure \(\overline{f^s(Z\setminus I_{f^s})}\) where \(I_{f^s}\) is the indeterminacy set of \(f^s\). 
Fix an integer \(N\) such that \(N\cdot \delta>2\).

We claim that \(f^s(Z)=f^t(Z)\) for some \(1\leq s<t\leq N\). 
Suppose to the contrary.
We then have \(\dim(f^s(Z)\cap f^t(Z))<\dim(Z)\) for any \(1\leq s<t\leq N\). 
We choose the positive number \(\varepsilon=\delta/(N-1)\). 
By the minimality of \(\dim(Z)\), for any sufficiently large \(i\) and for any \(1\leq s<t\leq N\), we have 
\(\#(P_{n_i}\cap f^s(Z)\cap f^t(Z))<\varepsilon(\# P_{n_i})\).
But then the following inequality for any sufficiently large \(i\) due to the inclusion-exclusion principle would give rise to a contradiction 
\[
\# P_{n_i}\ge (N\delta)(\# P_{n_i})-\frac{N(N-1)}{2}\cdot\varepsilon\cdot\# P_{n_i}>\# P_{n_i}.
\]
Hence, we have \(f^s(Z)=f^t(Z)\) for some \(s<t\) and thus \(Z\) is \(f\)-preperiodic. 
With \(Z\) replaced by \(f^{t-s}(Z)\), we may assume that \(Z\) is \(f\)-periodic. 
Here, a proper analytic subset \(Z\) is said to be \textit{\(f\)-periodic of period \(m\)}, if the closure \(\overline{f^m(Z\setminus I_{f^m})}\) coincides with \(Z\). 
We used the fact that $f^n$ doesn't contract $Z$ to a lower dimensional set as \(\dim (Z)\) is minimal.

Let \(m\) be the minimal period of \(Z\).
We claim that for any sufficiently large \(i\gg 1\), the number $n_i$ is divisible by $m$.
Suppose to the contrary.
Then replacing \(\{n_i\}\) by a subsequence, we may assume that the remainder \(c_{n_i}\equiv n_{i}\mod m\) is non-zero for any \(i\). 
As there are only finitely many choices of the remainders, replacing \(\{n_{i}\}\)  by a subsequence, we may further assume that all \(c_{n_i}\) are equal to a constant \(c_0\) with \(0<c_0<m\). 
But then, using that the points in \(P_n\) are fixed by \(f^n\), we obtain that \(P_n\cap Z\subseteq P_n\cap f^n(Z)\).
Hence,
\[
\# (P_{n_{i}}\cap Z\cap f^{c_0}(Z))=\# (P_{n_{i}}\cap Z\cap f^{n_{i}}(Z))=\# (P_{n_{i}}\cap Z)\ge \delta(\# P_{n_{i}}),
\]
a contradiction to the minimality of \(\dim(Z)\). 
So we have \(n_i=n_i'\cdot m\) for all sufficiently large \(i\) and with 
$f$ replaced by $f^m$ we may assume that $Z$ is \(f\)-invariant. 

By the choice of \(Z\),  for any proper analytic subset \(Z_0\subseteq Z\), we have $\# (P_{n_i}\cap Z_0)=o(\# P_{n_i})$ for all the sufficiently large $i$. 
Hence, together with Lemma \ref{l:singular-periodic} (applied for \(f_{|Z}\)) and Proposition \ref{p:small-dyn}, we obtain that 
\[
\limsup_{i\to\infty}\frac{1}{n_i}\text{log}\,\#(P_{n_i}\cap Z)\leq h_a(f_{|Z})\leq h_a(f).
\]
This leads to a contradiction to our assumption (\ref{eq-geq1}) by using (\ref{eq-geq2}) and we finish the proof of Theorem \ref{t:main2}.
\end{proof}

We end the section with the remarks below on special cases related to Theorem \ref{t:main2}.
\begin{remark}[Projective varieties over an algebraically closed field]\label{rem:other-alg-closed}
When \(f\colon X\to X\) is a dominant rational map of a projective variety of dimension \(k\) with only isolated periodic points, we can also obtain Theorem \ref{t:main2} since a sufficiently large multiple of the intersection cycle of ample divisors is greater than the cycle of the graph in \(X\times X\). 
This different approach works in any algebraically closed field. 
For the convenience of the reader, we give a sketch proof herein. 
Note that this approach does not work in our non-projective case (cf.~\cite[Remark 3.2]{DNT-BLMS}, \cite{Tru15} and \cite[Section 5]{Xie23}).

Assume that \(d_s(f)\) is the maximal dynamical degree. 
Denote by \(\Delta\) the diagonal of \(X\times X\) and \(\Gamma_n\) the graph of \(f^n\). 
Let \(\pi_i\colon X\times X\to X\) be the natural projections.
Take an ample divisor \(H\) on \(X\) and let \(D=\pi_1^*H+\pi_2^*H\). 
Then \(D\) is ample on \(X\times X\); consequently, there exists some sufficiently large \(a>0\) such that \(D^a\otimes \mathcal{I}_{\Delta}\) is globally generated, where \(\mathcal{I}_{\Delta}\) is the ideal sheaf of the diagonal. 
Therefore, we have \(\#\IP_n^0(X,f)\leq \#(\Gamma_n\cap\Delta)\leq a^k\Gamma_n\cdot D^{k}\) where the last inequality is due to the choice of \(a\) (cf.~\cite[Lemma 5.10]{Xie23}). 
Then
\begin{align*}
\Gamma_n\cdot D^k=\sum_{i=0}^k\binom{k}{i}(\Gamma_n\cdot \pi_1^*H^i\cdot \pi_2^*H^{k-i})=\sum_{i=0}^k\binom{k}{i}(H^i\cdot (f^n)^*H^{k-i}).
\end{align*}
By our assumption, for any \(\varepsilon>0\), there exists a constant \(c>0\) such that for any \(n\) and any \(0\leq i\leq k\), we have
\[
(f^n)^*H^{k-i}\cdot H^i\leq c(d_s(f)+\varepsilon)^n.
\]
Therefore, we obtain that
\[
\# \IP_n^0(X,f)\leq c'\cdot (d_s(f)+\varepsilon)^n,
\]
where \(c'=a^k\cdot c\cdot 2^k\), and then we conclude our result.
\end{remark}

\begin{remark}[Bimeromorphic maps of Fujiki surfaces]\label{rem-fujiki-surfaces}
In the joint paper \cite{DNT-BLMS} of the first author, there is also an estimate of isolated periodic points for a bimeromorphic self-map \(f\colon X\to X\) on a (smooth) compact K\"ahler surface \(X\) which is algebraically stable (in the sense of Forn{\ae}ss-Sibony), that is, \((f^n)^*=(f^*)^n\) as linear maps on \(H^{1,1}(X,\mathbb{C})\) for all \(n\in \mathbb{N}\); under this assumption, \cite[Theorem 1.4]{DNT-BLMS} asserts that if  \(d_1(f)>1\), then the number of isolated \(f\)-periodic points of period \(n\) counted with multiplicities satisfies 
\[
\#\widetilde{\IP}_n(f)\leq d_1(f)^n(1+o(1)),
\]as \(n\) tends to infinity. 
One can follow the same strategy as in the proof of Theorem \ref{t:main2} to show the following statement.

\((*)\) Let \(g\colon S\to S\) be a bimeromorphic map of a compact complex surface in the Fujiki class \(\mathcal{C}\).
Suppose that \(d_1(g)>1\).
Denote by \(P_n\coloneqq\IP_n^0(S\setminus I_g,g)\) the set of isolated \(g\)-periodic points of period \(n\) whose orbits do not intersect \(I_g\).
Then we have the cardinality 
\[\#P_n\leq d_1(g)^n(1+o(1)).\]
In fact, applying \cite[Theorem 0.1]{DF01}, we can take a further modification \(\widetilde{S}\) of a smooth K\"ahler model of \(S\) to obtain the commutative diagram
\[
\xymatrix{
\widetilde{S}\ar[r]^f\ar[d]_\pi&\widetilde{S}\ar[d]^\pi\\
S\ar[r]_g&S
}
\]
where \(f\) is an algebraically stable bimeromorphic map on a (smooth) compact K\"ahler surface \(\widetilde{S}\), and \(\pi\) is an isomorphism over a dense Zariski open subset \(U\).
By \cite[Theorem 1.4]{DNT-BLMS}, it is clear that \(\#(P_n\cap U)\leq d_1(g)^n(1+o(1))\) as \(n\) tends to infinity. 
For the complement \(Z\) of \(U\), which is a Zariski closed subset of \(S\), 
suppose to the contrary that \(\#(P_{n_i}\cap Z)\geq \delta d_1(g)^{n_i}\) for some positive number \(\delta>0\) and for a subsequence \(\{n_i\}\).
Then with \(Z\) replaced by an irreducible component (which is a curve), we may assume that \(\#(P_{n_i}\cap Z)\geq \delta d_1(g)^{n_i}\) and hence \(Z\not\subseteq I_g\).
By playing the same trick as in the proof of Theorem \ref{t:main2}, we may assume that \(Z\) is \(g\)-invariant.
But then the bimeromorphic map \(g\) would induce an automorphism on the curve \(Z\) with \(d_1(g_{|Z})=1\). 
This leads to a contradiction to Theorem \ref{t:main2} for \(g_{|Z}\) as \(\#(P_{n_i}\cap Z)\geq \delta d_1(g)^{n_i}\) by assumption.
So \((*)\) is thus verified.
\end{remark}

\section{Proofs of Theorem \ref{t:main1} and Corollary \ref{c:swz}}\label{s:proofs of conj}
In this section, we will prove Theorem \ref{t:main1} and Corollary \ref{c:swz}. 
In Subsection \ref{sub:lower bound}, we prove an equidistribution result for meromorphic maps with a dominant topological degree so as to establish the lower bound, and in Subsection \ref{sub:upper bound}, we follow a similar strategy of the proof of Theorem \ref{t:main2} to analyze the intersection of periodic points with any proper analytic subset so as to establish the upper bound.
Finally, the complete proofs of Theorem \ref{t:main1} and Corollary \ref{c:swz} will be given in Subsection \ref{sub:proof of mainthm}.

\subsection{Lower bound}\label{sub:lower bound}
In this subsection, we establish the lower bound in Theorem \ref{t:main1}.
As we are working for lower bound, we can remove a  suitable proper analytic subset so that we reduce to the equidistribution of repelling points with the orbits contained in an open subset of its compact K\"ahler model. 

\begin{theorem}\label{t:equi-mero}
Let \(f\colon X\to X\) be a dominant meromorphic map of a compact K\"ahler manifold \(X\), and $Y$ a proper analytic subset of $X$.
Let $P_n$ be one of the sets (cf.~Notation \ref{n:periodic}) $\IP_n$, $\widetilde \IP_n$, $\RP_n(X\setminus Y)$, and $\RP^0_n(X\setminus Y)$. 
Suppose that \(f\) has a dominant topological degree \(d_t\). 
Then we have  
$\# P_n= d_t^n(1+o(1))$ when $n$ goes to infinity. Moreover, the points in $P_n$ are equidistributed with respect to the equilibrium measure $\mu$ of \(f\), i.e.,
$$\lim_{n\to\infty} \frac{1}{d_t^n} \sum_{a\in P_n} \delta_a = \mu,$$
where $\delta_a$ denotes the Dirac mass at $a$.
\end{theorem} 

\begin{proof}
The result is known for $\IP_n$, $\widetilde \IP_n$ and \(\RP_n\) by \cite[Theorem 1.1]{DNT-ind}.
So it is enough to prove that there is a subset $Q_n\subseteq \RP_n^0(X\setminus Y)$ such that 
$$\lim_{n\to\infty} \frac{1}{d_t^n} \sum_{a\in Q_n} \delta_a = \mu.$$
For this purpose, we can enlarge the subset $Y$. 
Denote by $\Gamma$ the closure of the graph of $f$ in $X\times X$ and $\pi_1,\pi_2$ the canonical projections from $X\times X$ to $X$ respectively. 
The first (resp. second) indeterminacy sets $I_f$ (resp. $I'_f$) of $f$ are defined by the set of points over which the fiber of $\pi_{1|\Gamma}$ (resp. $\pi_{2|\Gamma}$) is of positive dimension. 
Note that the fibre of \(\pi_{1|\Gamma}\) over a point outside the first indeterminacy locus is a singleton. 
Let \(I_0\coloneqq I_f\), \(I_{n+1}\coloneqq I_0\cup f(I_n)\) and \(I_{\infty}\coloneqq\cup_{n\ge 0}I_n\), and let \(I_0'\coloneqq I_f'\), \(I_{n+1}'\coloneqq I_0'\cup f(I_n')\) and \(I_{\infty}'\coloneqq\cup_{n\ge 0}I_n'\).
Note that here the image \(f(A)\coloneqq\pi_2(\pi_1^{-1}(A)\cap\Gamma)\) of a subset \(A\subseteq X\) contains (but not necessarily coincides with) the closure \(\overline{f(A\setminus I_f)}\). 

We shall follow the strategy as in \cite[Section 3]{DNT-ind} to construct, for generic small balls, an almost maximal number of inverse branches with respect to \(f^n\), which are disjoint with \(Y\) and whose sizes we control. 
We briefly explain the strategy for the convenience of the reader.

Choose an analytic subset \(\Sigma_0\) of \(X\) containing the union of \(Y,  I_f, I'_f,f(I_f)\) and \(f^{-1}(I'_f)\) such that the restriction of \(\pi_2\) to \(\Gamma\setminus\pi_2^{-1}(\Sigma_0)\) defines an unramified cover over \(X\setminus\Sigma_0\). 
Let \(B\) be a connected subset of \(X\). 
A continuous bijective map \(g\colon B\to B_{-n}\) with \(B_{-n}\subseteq X\) is said to be
 \textit{an inverse branch of \(B\) with order \(n\)}, if  
 \((B_{-i}\coloneqq f^{n-i}(B_{-n}))\cap\Sigma_0=\varnothing\) for \(0\leq i\leq n\), \(f\colon B_{-i}\to B_{-i+1}\) is a bijective map for \(1\leq i\leq n\), \(B_0=B\), and \(f^n\circ g=\text{id}\) on \(B\).  
Note that \(B\) admits at most \(d_t^n\) inverse branches of order \(n\). 
The condition \(B_{-i}\cap \Sigma_0=\varnothing\) implies that the inverse branch can be extended to any small enough open set containing \(B\), which is also disjoint with \(\Sigma_0\). 
We also call \(B_{-n}\) the image of the inverse branch \(g\colon B\to B_{-n}\).

By the same construction as in \cite[Pages 1813-1814]{DNT-ind}, we fix any constant \(\theta\in (d_{k-1}/d_t,1)\) and then choose a sufficiently large integer \(N\).  
Let \(\Sigma\coloneqq\cup_{0\leq i\leq N}f^i(\Sigma_0)\). 
Choose a suitable resolution \(\pi\colon\widehat{\Gamma}\to\Gamma\) and 
define \(\tau_i\coloneqq\pi_i\circ\pi\) such that \(\tau_1^{-1}(\Sigma)\) and \(\tau_2^{-1}(\Sigma)\) are of pure codimension 1 in \(\widehat{\Gamma}\). 
Note that $\Sigma$ itself is not necessarily a hypersurface; it may have components of different dimensions.  
Fix a sufficiently large K\"ahler form \(\widehat{\omega}\geq \tau_i^*\omega\) on \(\widehat{\Gamma}\) (cf.~\cite{Blanchard-ens}). 
Denote by \(\Sigma'\) and \(\Sigma''\) the union of components of codimension 1 and the union of components of codimension \(\geq 2\) of \(\Sigma\) respectively.
Define
\[
S_0\coloneqq c_1\theta^{-N}d_t^N([\Sigma']+(\tau_2)_*(\widehat{\omega})),~~S\coloneqq\sum_{n=0}^N(f_*)^n(S_0),~~R\coloneqq8\sum_{m\ge 0}(\theta d_t)^{-mN}(f^N)_*^m(S).
\]
As in \cite[p.1814]{DNT-ind}, when \(c_1\) is sufficiently large,  the above positive closed $(1,1)$-current $S_0$ has Lelong number $\geq 1$ at every point of $\Sigma$ (see \cite[Lemma 3.2]{DNT-ind}). 
This current \(S_0\) dominates the ``obstruction'' to construct inverse branches. 
It allows us to construct inverse branches of balls as in \cite[Proposition 3.1]{DNT-ind} with an additional property that $B_{-i}\cap Y=\varnothing$ for every $i$. 
We then follow \cite[Page 1824, End of the proof of Theorem 1.1]{DNT-ind} to get a lower bound of the repelling periodic points of period $n$ whose orbits do not meet $Y$ nor the indeterminacy sets \(I_\infty\cup I_{\infty}'\), noting that the equilibrium measure of \(f\) has no mass on any proper analytic subset. So our theorem is thus proved. 
\end{proof}

\subsection{Upper bound}\label{sub:upper bound}
In this subsection, we follow a similar strategy of the proof of Theorem \ref{t:main2} to control the growth of the intersection of periodic points with any proper analytic subset and thus we can give an upper bound. 
\begin{proposition}\label{p:small-per}
Let \(f\colon X\to X\) be a dominant meromorphic map of a Fujiki variety that has a dominant topological degree, \(I_f\) the indeterminacy locus of \(f\), and $Z$ a proper analytic subset of $X$ which is not contained in \(I_f\). 
Then we have the cardinality $\# \IP_n^0(Z\setminus I_f)=o(\# \IP_n^0(X\setminus I_f))$ as \(n\) tends to infinity.
\end{proposition}

\begin{proof}
For the convenience, we let \(P_n\coloneqq \IP_n^0(X\setminus I_f,f)\). 
Suppose to the contrary. 
Then  there is a number $\delta>0$ and an increasing sequence $\{n_i\}$ such that $\# (P_{n_i}\cap Z)\geq \delta (\# P_{n_i})$ for every $i$. 
Besides, replacing \(Z\) with an irreducible analytic subset and replacing \(\delta\) with another positive number (possibly smaller than \(\delta\)), we may assume that $Z$ has the minimal dimension satisfying the above inequality.
In other words, there is a proper analytic subset \(Z\) of minimal dimension such that for each \(i\)
\[
\#(P_{n_i}\cap Z)\geq \delta \# P_{n_i}.
\]
With the same argument as in the proof of Theorem \ref{t:main2}, we may assume that \(Z\) is \(f\)-invariant.
Therefore, it follows from Theorem \ref{t:equi-mero} that
\begin{align*}
\delta d_t^{n_i}(1+o(1))\leq\delta(\# P_{n_i})\leq\# (P_{n_i}\cap Z).
\end{align*}
Hence, we obtain that
\[
\text{log}\,d_t\leq \limsup_{i\to\infty}\frac{1}{n_i}\text{log}\,\#(P_{n_i}\cap Z)\leq h_a(f_{|Z}),
\]
where the last inequality is due to Theorem \ref{t:main2}, noting that by our assumption on \(\dim(Z)\), the subset \(Z\) cannot be contracted to a lower dimensional analytic subset via \(f\).
But then this contradicts  Proposition \ref{p:small-dyn} as the maximal dynamical degree of \(f_{|Z}\) is no more than \(d_{k-1}\).
Our proposition is thus proved.
\end{proof}

\subsection{Proofs of the statements}\label{sub:proof of mainthm}
We shall conclude our paper by proving Theorem \ref{t:main1} and Corollary \ref{c:swz}.
\begin{proof}[Proof of Theorem \ref{t:main1}]
Let \(\pi\colon\widetilde{X}\to X\) be a blowup from a compact K\"ahler manifold \(\widetilde{X}\) to the Fujiki variety \(X\), and \(\widetilde{f}\) the dominant meromorphic map on \(\widetilde{X}\) induced by \(f\).
Let \(U\) be the open subset of \(X\), over which \(\pi\) is an isomorphism.
Let \(Y'\coloneqq Y\cup (X\setminus U)\).

First, we prove the lower bound.
Let us take \(\widetilde{Y}=\pi^{-1}(Y')\), which is a proper analytic subset of \(\widetilde{X}\).
Applying Theorem \ref{t:equi-mero} to the pair \((\widetilde{X},\widetilde{Y})\), we have
\[
\# \IP_n^0(X\setminus Y,f)\geq \# \IP_n^0(X\setminus Y',f)\ge\# \RP_n^0(\widetilde{X}\setminus \widetilde{Y},\widetilde{f})=d_t^n(1+o(1))
\]
as \(n\) tends to infinity.

Second, we prove the upper bound.
Since \(\pi\) is an isomorphism over the open subset \(U\),
applying Theorem \ref{t:equi-mero} to \(\widetilde{X}\), we have
\[
\# \IP_n^0(X\setminus Y',f)\leq   \# \IP_n(\widetilde{X}\setminus \widetilde{Y},\widetilde{f})\leq\#\widetilde \IP_n(\widetilde{f})=d_t^n(1+o(1))
\]
as \(n\) tends to infinity.
On the other hand, applying Proposition \ref{p:small-per} for \(Y'\),
we get $\#\IP_n^0(Y'\setminus I_f,f)=o(\# \IP_n^0(X\setminus I_f,f))$. 
So our theorem is thus proved.
\end{proof}

\begin{proof}[Proof of Corollary \ref{c:swz}] 
Let \(k\coloneqq \dim(X)\).
First, we show that the \(f\)-periodic points of period \(n\) are isolated. 
Suppose to the contrary.
Then with \(f\) replaced by a power, we can find a positive dimensional \(f\)-invariant analytic subvariety \(Y\) such that \(f|_{Y}=\text{id}\).
On the other hand, the restriction of the K\"ahler class \([\omega]\) to \(Y\) is also a K\"ahler class which we denote by \([\omega_Y]\) (see~\cite[Proposition 3.6]{GK-AIF} and \cite{Vaj96}). 
Then we obtain that \((f|_Y)^*[\omega_Y]=q[\omega_Y]\), which leads to a contradiction to \(f|_{Y}=\text{id}\).

In view of Theorem \ref{t:main1}, we only need to prove that \(f\) has a dominant topological degree. 
Recall that the K\"ahler class \([\omega]\) is represented by a smooth K\"ahler form \(\omega\) and 
we can do the wedge product of \(\omega\) along the smooth locus of \(X\). 
Moreover, as \(f^*[\omega]=[f^*\omega]=[q\omega]\), there is a globally smooth function \(u\) on \(X\) such that \(f^*\omega=q\omega+i\partial\overline{\partial}u\) (see Definition \ref{d:differential}; cf.~\cite[Remark 3.2]{HP16}). 
Let \(\pi\colon\widetilde{X}\to X\) be a resolution of singularities. 
Then the pullback \(\alpha\coloneqq\pi^*\omega\) is also a smooth \((1,1)\)-form on \(\widetilde{X}\) (see \cite[Lemma 1.3]{Dem85} and Definition \ref{d:differential}) and the cohomology class \([\alpha]=[\pi^*\omega]=\pi^*[\omega]\) is a nef and big class (see \cite[Proposition 2.6 (2) and (4)]{Zhong-Asian} and \cite[Proposition 3.6]{GK-AIF}), noting that the resolution factors through the normalization of \(X\) which is a finite bimeromorphic map. 
This implies that 
\[
\int_{X_\reg}\omega^k=\int_{\widetilde{X}}\pi^*\omega^k>0.
\]
Note that the first integration is over the smooth locus of \(X\) and the second integration is the volume of a nef and big class which is positive (see \cite[Theorem 0.5]{DP-annals}).
Also, via the pullback, we have \(\pi^*f^*\omega=q\alpha+i\partial\overline{\partial}(u\circ \pi)\); see \cite[\S 1.2]{Varouchas-mathann}.  
Hence, it follows from the Stokes’ theorem that 
\begin{align}\label{eq-integral}
q^k\int_{X_\reg}\omega^k=\int_{\widetilde{X}}(q\alpha)^k=\int_{\widetilde{X}}(\pi^*f^*\omega)^k=\int_{X_\reg}(f^*\omega)^k=\deg(f)\cdot\int_{X_\reg}\omega^k
\end{align}
where the last equality is due to the projection formula (see e.g. \cite[Proposition 2.2]{Zhong-Asian}). 
In particular, we obtain that \(d_t=\deg(f)=q^k\). 

Next, we show that \(d_1(f)=q\).
Let \(\widetilde f\) be the induced dominant meromorphic map on \(\widetilde X\). 
Consider the following commutative diagram 
\[\xymatrix{
&\widetilde \Gamma_n\ar[dr]^{p_2^{(n)}}\ar[dl]_{p_1^{(n)}}&\\
\widetilde{X}\ar[rr]^{\widetilde{f}^n}\ar[d]_{\pi}&&\widetilde{X}\ar[d]^{\pi}\\
X\ar[rr]_{f^n}&&X}
\]
where \(f^n\) is holomorphic, \(\widetilde f^n\) is a dominant meromorphic map such that \(\pi\circ\widetilde f^n=f^n\circ\pi\), and \(\widetilde \Gamma_n\) is a resolution of singularities of the graph of \(\widetilde f^n\).
Then we have
\[(\widetilde{f}^n)^*(\alpha)=(\widetilde{f}^n)^*\pi^*(\omega)=(p_1^{(n)})_*(p_2^{(n)})^*\pi^*(\omega)=\pi^*(f^n)^*(\omega)=q^n(\alpha+i\partial\overline{\partial}u_n),\]
where the last second equality is due to the projection formula (cf.~\cite[Proposition 2.2]{Zhong-Asian}), and 
\[u_n=\sum_{i=0}^{n-1}q^{n-1-i}u\circ f^i\circ\pi\] 
is a globally smooth function. 
Therefore, 
\[
d_1(f)\coloneqq d_1(\widetilde{f})=\lim_{n\to\infty}\left(\int_{\widetilde{X}}(\widetilde{f}^n)^*\alpha\wedge \alpha^{k-1}\right)^{\frac{1}{n}}=\lim_{n\to\infty}\left(\int_{\widetilde{X}}q^n\alpha\wedge \alpha^{k-1}\right)^{\frac{1}{n}}=q,
\]
where the last second equality is by the Stokes theorem.
On the other hand, by the log-concavity of dynamical degrees (see Definition \ref{d:dynamical-degree}), we have \(d_p(\widetilde{f})\leq d_1(\widetilde{f})^p=q^p\) for each \(p\leq k\). 
As the topological degree satisfies  \(d_t(\widetilde{f})=q^k\), we deduce that \(d_p(\widetilde{f})=q^p\) for every \(p\). 
This implies that \(\widetilde{f}\) and hence \(f\) have a dominant topological degree.
We complete the proof of the corollary.
\end{proof}


\end{document}